\RequirePackage{fix-cm}
\documentclass[smallextended]{svjour3}       
\smartqed  
\usepackage{graphicx}
\usepackage{amsmath}
\usepackage{amssymb, comment}
\usepackage{mathrsfs}
\usepackage{multirow}
\usepackage[colorlinks=true]{hyperref}
\numberwithin{equation}{section}

\def \o {\omega}

\def \E {\mathbb E}

\def \P {\mathbb{P}}

\def \approxd {\,{\buildrel d \over \approxd}\,}

\usepackage{subcaption}

\renewenvironment{proof}[1][\proofname]{\par \normalfont \trivlist
 \item[\hskip\labelsep\itshape #1]\ignorespaces
}{
 \hspace*{\fill}$\Box$ \endtrivlist
}
\renewcommand{\proofname}{\noindent {\bf Proof}}

\title{Applications of the $M^X/\text{semi-Markov}/1$ queue to road traffic}

\begin{document}

\author{Abhishek \and
        Marko~Boon \and
        Rudesindo~N\'u\~nez-Queija
        }

\institute{Abhishek $\&$  Rudesindo~N\'u\~nez~Queija \at
           Korteweg-de Vries Institute for Mathematics, University of Amsterdam, Amsterdam, The Netherlands\\
           \email{\{Abhishek, nunezqueija\}@uva.nl}
         \and
        Marko~Boon \at
        Department of Mathematics and Computer Science, Eindhoven University of Technology, Eindhoven, The Netherlands \\
        \email{m.a.a.boon@tue.nl}
       }

\date{\today}
\maketitle
\begin{abstract}
The single server queue with multiple customer types and semi-Markovian service times, sometimes referred to as the $M/SM/1$ queue, has been well-studied since its introduction by Neuts in 1966. In this paper, we apply an extension of this model, with batch arrivals and exceptional first service, to road traffic situations involving multiple streams of conflicting traffic. In particular, we use it in the context of \emph{gap acceptance models} where low-priority traffic needs to cross (or, depending on the application, merge with) another traffic flow of higher priority.\\

Traditionally, gap acceptance models are based on the $M/G/1$ queue with exceptional first service, in this application area commonly referred to as the $M/G2/1$ queue. In an earlier study \cite{AbhishekMergingModel}, we showed how the $M^X/SM/1$ queue with exceptional first service can be applied in this context to extend the model with driver impatience and more realistic merging behaviour. In this paper, we show how this same queueing system can be used to model a Markov modulated Poisson arrival process of the high-priority traffic stream. Due to its flexibility, this arrival process is very relevant in this application, particularly because it allows the modelling of platoon forming of vehicles.
The correlated inter-arrival times of these high priority vehicles cause the merging times of two subsequent low priority vehicles to become dependent as well (as they correspond with the service times in the underlying queueing model).
We derive the waiting time and sojourn time distributions of an arbitrary customer, showing that these depend on the position of the customer inside the batch, as well as on the type of the first customer in the batch.

\end{abstract}

\section{Introduction}
The single server queue with multiple customer types and semi-Markovian service times, sometimes referred to as the $M/SM/1$ queue, has been well-studied  since its introduction by Neuts \cite{neuts66}. An overview of the earlier existing literature \cite{AbhishekHeavyTraffic,cin_s,gaver,neuts77a,neuts77b} can be found in 
\cite{QUESTA2017}, in which the transient and stationary queue length distributions in a single server model with batch arrivals and semi-Markov service times were analyzed. In this paper, we apply an extension of this model, with batch arrivals and exceptional first service, to road traffic situations involving multiple streams of conflicting traffic \cite{drew1,drew2,drew3,wei,wu2001}. In particular, we use it in the context of \emph{gap acceptance models} \cite{abhishekcomsnets2016,AbhishekMergingModel,heid94,heid97} where low-priority traffic needs to cross (or, depending on the application, merge with) another traffic flow of higher priority. Drivers in the low-priority traffic flow wait until a sufficiently large gap arises between two subsequent vehicles in the high-priority traffic flow.
This minimal gap, which may be vehicle-specific, is commonly referred to as the \emph{critical headway} denoted by $T$.\\

In the gap acceptance literature, three variations of driver behavior are distinguished (cf. \cite{AbhishekMMPPandImpatience,heid97}). With the first behavior type (referred to as B$_1$ in this paper), all low-priority vehicle drivers require the same constant critical headway to merge with the high-priority stream of vehicles. In the second case (B$_2$), which is commonly referred to as \emph{inconsistent} gap acceptance behavior, each low-priority driver samples a critical headway from a given distribution at each new attempt. Its natural counterpart is known as \emph{consistent} gap acceptance behavior (B$_3$), where the low-priority driver samples a critical headway from a given distribution only for his first attempt and then uses the same value at his subsequent attempts. \\

Traditionally, gap acceptance models are based on the $M/G/1$ queue with exceptional first service (cf. \cite{welch,yeo,yeoweesakul}), in this application area commonly referred to as the $M/G2/1$ queue.
The ``service times'' correspond to the time required to search for a sufficiently large gap and crossing the intersection or, depending on the application, merging with the high-priority traffic flow.
In an attempt to make the standard gap acceptance model more realistic, \cite{AbhishekMergingModel}, we
developed a general framework  based on the $M/SM/1$ queue with batch arrivals and exceptional first service (which we refer to as the $M^X/SM2/1$ queue; see \cite{QUESTA2017,AbhishekHeavyTraffic}) and
showed how this queueing model can be applied in this context to extend the standard gap acceptance model with driver impatience and more realistic merging behaviour. In the present paper, we show how to exploit the versatility of the $M^X/SM2/1$ queue to model a Markov modulated Poisson arrival process of the high-priority traffic stream. Due to its flexibility, this arrival process is very relevant in this application, particularly because it allows the modelling of platoon forming of vehicles. We refer the reader to 
\cite{AbhishekMMPPandImpatience} for a brief overview of the earlier existing literature relevant to it.
The correlated inter-arrival times of these high-priority vehicles cause the merging times of two subsequent low priority vehicles to become dependent as well (as they correspond with the service times in the underlying queueing model).\\

The contributions of this paper are twofold. First, we show how to derive the waiting time and sojourn time distributions of an arbitrary customer for the $M^X/SM/1$ queueing system with exceptional first service, showing that these depend on the position of the customer inside the batch, as well as on the type of the first customer in the batch. Second, we focus on the application of this queueing model to road traffic situations involving multiple conflicting traffic streams, where on the minor road, vehicles arrive in batches according to a Poisson process and the arrival process on the major road is a Markov modulated Poisson process. Based on numerical examples, we demonstrate the impact of the three types of the driver behavior (B$_1$, B$_2$ and B$_3$), on the delay on the minor road. More specifically, we show that the expected waiting times for the all three behavior types depend not only on the mean batch size, but also on the full distribution of the batch sizes. \\

The remainder of this paper is organized as follows. In Section \ref{description}, we present the description of the queueing model. Using  the results from \cite{AbhishekHeavyTraffic}, we obtain the LST (Laplace-Stieltjes transform) of the steady-state waiting time and sojourn time distributions of customers as well as batches in Section \ref{stationay waiting time}. In Section \ref{App to road traffic}, we first give several applications in which the extended queueing model arises, and then study the application to road traffic situations involving multiple conflicting traffic streams. In Section \ref{numericalresults:ch5}, we present the numerical examples.

\section{The $M^X/SM2/1$ queueing model}\label{sect:queueingmodel}
In this section, we first describe the $M^X/SM2/1$ queuing model. Subsequently, we use the results from our paper \cite{AbhishekHeavyTraffic} on the steady-state distribution of the queue length, to derive the waiting time and sojourn time distributions.
\subsection{Model description}\label{description}
 Customers arrive in batches  at a single-server queuing system according to a Poisson process  with intensity $\lambda$. The arriving batch size is denoted by the random variable $B$, with probability generating function (PGF) $B(z)$, for $|z|\leq 1$ (zero-sized batches are not allowed, i.e. $B\geq 1$). Customers are served individually and the service process is considered as a semi-Markov (SM) process. In addition, we assume that the first customer in each busy period has a different service time distribution  than regular customers served in the busy period such that, for $i=1,2,\dots,N$,
\begin{align}
\tilde{G}_{ij}(s)&=\E[e^{-sG^{(n)}}1_{\{J_{n+1}=j\}}|J_n=i, X_{n-1}\geq 1],\label{G_{ij}(s)}\\
\tilde{G}^*_{ij}(s)&=\E[e^{-sG^{(n)}}1_{\{J_{n+1}=j\}}|J_n=i, X_{n-1}=0]\label{G*_{ij}(s)},
\end{align} where $J_n$ is the type of the $n$-th customer and $G^{(n)}$ is its service time,
 and $X_{n-1}$ is the number of customers in the system at the departure of the $(n-1)$-th customer. \\

In particular, for $i,j=1,2,\dots,N$, we define
\begin{align}
P_{ij}=\tilde{G}_{ij}(0)=\P(J_{n+1}=j|J_n=i,X_{n-1}\geq 1),\label{p_ij}\\
P^*_{ij}=\tilde{G}^*_{ij}(0)=\P(J_{n+1}=j|J_n=i,X_{n-1}= 0).\label{p^*_ij}
\end{align}

To be consistent with the terminology used in the gap acceptance literature, we refer to this queueing system as the $M^X/SM2/1$ queue.
In this section, for improved readability, we briefly sketch the proof in \cite{AbhishekHeavyTraffic} to obtain the PGF of the queue length distribution at departure times of customers, which will be used to derive the waiting time and sojourn time distributions in the next section.\\

The queue length distribution at departure times can be obtained using the following recurrence relation:
\begin{align}
X_n = \left\{
\begin{array}{l l}
\ X_{n-1}-1+A_n & \quad \text{if $X_{n-1} \geq 1$ }\\
A_n +B_n-1& \quad \text{if $X_{n-1} =0$}
\end{array} \right., ~~~ n=1,2,3,\dots,
\label{recurA}
\end{align} where $A_{n}$ is the number of arrivals during the service time of the $n$-th customer
, and $B_n$ is the size of the batch in which $n$-th customer arrived, with PGF $B(z)$, for $|z|\leq 1$.
The conditional PGFs of the  queue length distribution at departure epochs are obtained by solving the following system of $N$ equations:
\begin{align}
& (z-A_{jj}(z))f_j(z)-\sum_{i=1,i\neq j}^{N}A_{ij}(z)f_i(z)
=\sum_{i=1}^{N}(B(z)A^{*}_{ij}(z)-A_{ij}(z))
f_i(0),\quad \quad ~~~~~~ j=1,2,\dots,N, \label{twentysevenA}
\end{align}
where $f_i(z)= {\rm lim}_{n \rightarrow \infty} \E[z^{X_n}1_{\{J_{n+1}=i\}}]$, $A_{ij}(z)=\E[z^{A_{n}}1_{\{J_{n+1}=j\}}|J_n=i,X_{n-1}\geq 1]$, $A^{*}_{ij}(z)=\E[z^{A_{n}}1_{\{J_{n+1}=j\}}|J_n=i,X_{n-1}=0]$, and hence the PGF of the  queue length distribution at departure epochs is given by
\begin{equation}
F(z)=\sum_{i=1}^{N}f_i(z). \label{F(z)}
\end{equation}
As customers arrive in the system according to a batch Poisson process with rate $\lambda$, for $i,j=1,2,\dots,N$, we obtain,
\begin{align}
  A_{ij}(z)=&\tilde{G}_{ij}(\lambda(1-B(z))),\label{rel_A_G} \\
  A^{*}_{ij}(z)=&\tilde{G}^{*}_{ij}(\lambda(1-B(z))).\label{rel_A*_G*}
\end{align}

\subsection{Waiting time and sojourn time}\label{stationay waiting time}
In this section, we shall determine the  waiting time and sojourn time distributions of an arbitrary batch as well as an arbitrary customer, noticing that the waiting time and sojourn time of a customer depend on its position in the batch, as well as on the type of service of the first customer in its batch.\\

%
To determine the waiting times and sojourn times of customers, firstly, we modify our model in such a way that  all customers in the same batch are served together as a {\em super customer}. Let $\mathcal{G}^{(n)}$ and $\mathcal{J}_n$ be the service time and the service type of the $n$-th  super customer respectively. Then, the LST of the conditional service time of a  super customer is defined as, for $\text{Re}(s)\geq 0, i,j=1,2,\dots,N$,
\begin{align}
\tilde{\mathcal{G}}_{ij}(s)&=\E[e^{-s \mathcal{G}^{(n)}}1_{\{\mathcal{J}_{n+1}=j\}}|\mathcal{J}_{n}=i, X_{n-1}\geq 1], \label{LST super_ij}\\
\tilde{\mathcal{G}}^{*}_{ij}(s)&=\E[e^{-s \mathcal{G}^{(n)}}1_{\{\mathcal{J}_{n+1}=j\}}|\mathcal{J}_{n}=i, X_{n-1}=0]. \label{LST super_ij*}
\end{align}

Now, we can obtain the LST of the conditional service time of a super customer in terms of the LST of the conditional service time of an individual customer as
\begin{align}
\tilde{\mathcal{G}}_{ij}(s)&=\E\left[[\tilde{\bold{G}}(s)^B]_{ij}\right], \\
\tilde{\mathcal{G}}^{*}_{ij}(s)&=\sum_{k=1}^{N}\tilde{G}^{*}_{ik}(s)\E\left[[\tilde{\bold{G}}(s)^{(B-1)}]_{kj}\right], \quad i,j=1,2,\dots,N,
\end{align} where $\tilde{\bold{G}}(s)=[\tilde{G}_{ij}(s)]$ is a matrix of order $N \times N$, and $[\tilde{\bold{G}}(s)^B]_{ij}$ is the $(i,j)$th element of matrix $\tilde{\bold{G}}(s)^B,$ for $i,j=1,2,\dots, N$. \\

Let $\mathcal{X}^{d}_n, \mathcal{X}^{bs}_n$ be the number of super customers in the queue at the departure of, and  the beginning of service of the $n$-th super customer respectively. We can derive the PGF of the number of super customers in the  queue, in steady state, at the departure of a super customer by letting $A_{ij}(z)=\tilde{\mathcal{G}}_{ij}(\lambda(1-z))$, $A^*_{ij}(z)=\tilde{\mathcal{G}}^{*}_{ij}(\lambda(1-z))$ and $B(z)=z$ in Equation \eqref{twentysevenA}.\\

Therefore, now, we know the distribution of the number of super customers at the departure of the super customer. But, to determine the waiting time of a super customer, using the distributional form of Little's law, we need to find the distribution of the number of super customers at the beginning of the service of a  super customer.  \\

We can write
\begin{align*}
\mathcal{X}^{bs}_n=\begin{cases}
\mathcal{X}^d_{n-1}-1, \quad &\text{if }\mathcal{X}^d_{n-1}\geq 1,\\
0, \quad &\text{if }\mathcal{X}^d_{n-1}=0.\\
\end{cases}
\end{align*}

This implies that
\begin{align}
\E[z^{\mathcal{X}^{bs}_n}1_{\{\mathcal{J}_n=i\}}]=&\E[z^{\mathcal{X}^{bs}_n}1_{\{\mathcal{J}_n=i\}} 1_{\{\mathcal{X}^d_{n-1}=0\}}]+\E[z^{\mathcal{X}^{bs}_n}1_{\{\mathcal{J}_n=i\}} 1_{\{\mathcal{X}^d_{n-1}\geq 1\}}],
\label{queue_BS}
\end{align} where
\begin{align}
\E[z^{\mathcal{X}^{bs}_n}1_{\{\mathcal{J}_n=i\}} 1_{\{\mathcal{X}^d_{n-1}=0\}}]
=&\P(\mathcal{X}^d_{n-1}=0,\mathcal{J}_n=i), \label{queue_BS_empty}
\end{align} and
\begin{align}
\E[z^{\mathcal{X}^{bs}_n}1_{\{\mathcal{J}_n=i\}} 1_{\{\mathcal{X}^d_{n-1}\geq 1\}}]
&=\E[z^{\mathcal{X}^{d}_{n-1}-1}1_{\{\mathcal{J}_n=i\}}]-\E[z^{\mathcal{X}^d_{n-1}-1}1_{\{\mathcal{J}_n=i\}}1_{\{\mathcal{X}^d_{n-1}=0\}}]\nonumber\\
&=\frac{1}{z}\left(\E[z^{\mathcal{X}^{d}_{n-1}}1_{\{\mathcal{J}_n=i\}}]-\P(\mathcal{X}^d_{n-1}=0,\mathcal{J}_n=i)\right). \label{queue_BS_nonempty}
\end{align}

 Let $\mathcal{W}^{sc}_n$ and $\mathcal{S}^{sc}_n$ be  the waiting time and sojourn time of the $n$-th  super customer respectively. By the distributional form of Little's law, we obtain
 \begin{align*}
  \E[z^{\mathcal{X}^{d}_{n}}1_{\{\mathcal{J}_{n+1}=i\}}]&=\E[e^{-\lambda(1-z)\mathcal{S}^{sc}_{n}}1_{\{\mathcal{J}_{n+1}=i\}}], \\ \E[z^{\mathcal{X}^{bs}_n}1_{\{\mathcal{J}_n=i\}}1_{\{\mathcal{X}^d_{n-1}=0\}}]&=\E[e^{-\lambda(1-z)\mathcal{W}^{sc}_n}1_{\{\mathcal{J}_n=i\}}1_{\{\mathcal{X}^d_{n-1}=0\}}],\\
   \E[z^{\mathcal{X}^{bs}_n}1_{\{\mathcal{J}_n=i\}}1_{\{\mathcal{X}^d_{n-1}\geq 1\}}]&=\E[e^{-\lambda(1-z)\mathcal{W}^{sc}_n}1_{\{\mathcal{J}_n=i\}}1_{\{\mathcal{X}^d_{n-1}\geq 1\}}].
 \end{align*}
 Letting $s=\lambda(1-z)$, then yields
 \begin{align}
\E[e^{-s\mathcal{S}^{sc}_{n}}1_{\{\mathcal{J}_{n+1}=i\}}]&= \E\left[\left(1-\frac{s}{\lambda}\right)^{\mathcal{X}^{d}_{n}}1_{\{\mathcal{J}_{n+1}=i\}}\right],\label{sojourn_super_type}\\
\E[e^{-s\mathcal{W}^{sc}_n}1_{\{\mathcal{J}_n=i\}}1_{\{\mathcal{X}^d_{n-1}=0\}}]&= \E\left[\left(1-\frac{s}{\lambda}\right)^{\mathcal{X}^{bs}_n}1_{\{\mathcal{J}_n=i\}}1_{\{\mathcal{X}^d_{n-1}=0 \}}\right],\label{waiting_super_empty}\\
\E[e^{-s\mathcal{W}^{sc}_n}1_{\{\mathcal{J}_n=i\}}1_{\{\mathcal{X}^d_{n-1}\geq 1\}}]&= \E\left[\left(1-\frac{s}{\lambda}\right)^{\mathcal{X}^{bs}_n}1_{\{\mathcal{J}_n=i\}}1_{\{\mathcal{X}^d_{n-1}\geq 1\}}\right]\label{waiting_super_nonempty}.
 \end{align}
 Subsequently, we obtain
  \begin{align}
\E[e^{-s\mathcal{S}^{sc}}]&= \E\left[\left(1-\frac{s}{\lambda}\right)^{\mathcal{X}^{d}}\right],\label{sojourn_supercustomer}\\
\E[e^{-s\mathcal{W}^{sc}}]&= \E\left[\left(1-\frac{s}{\lambda}\right)^{\mathcal{X}^{bs}}\right],\label{waiting_supercustomer}
 \end{align} where $\mathcal{S}^{sc}=\lim_{n\to \infty}\mathcal{S}^{sc}_n,\mathcal{W}^{sc}=\lim_{n\to \infty}\mathcal{W}^{sc}_{n},\mathcal{X}^{d}=\lim_{n\to \infty}\mathcal{X}^{d}_{n},\mathcal{X}^{bs}=\lim_{n\to \infty}\mathcal{X}^{bs}_n$.\\

Finally, we obtain the waiting times and sojourn times of individual customers in the batches by conditioning on the position of the customer in the batch, and using the following relations:
\begin{itemize}
\item the waiting time of the first customer in the batch is equal to the waiting time of the  super customer,
\item the waiting time of the $m$-th customer in the batch, for $m>1$, is equal to the waiting time of the super customer plus the service times of the first $(m-1)$ customers in the batch,
\item the sojourn time of the $m$-th customer in the batch, for $m<B$, is equal to the waiting time of the $(m+1)$-th customer,
\item the sojourn time of the last customer in the batch is equal to the sojourn time of the super customer.
\end{itemize}

Let $W^{(m)}$ and $S^{(m)}$ be the steady-state waiting time and sojourn time of the $m$-th customer served in his batch, respectively. Using the aforementioned relations, we obtain
\begin{align}
\E[e^{-sW^{(1)}}]&=\E[e^{-s\mathcal{W}^{sc}}],\\
\E[e^{-sW^{(m)}}]&=\sum_{k=1}^N\sum_{j=1}^{N}\sum_{i=1}^{N}\E[e^{-s\mathcal{W}^{sc}_n}1_{\{\mathcal{J}_n=i\}}1_{\{\mathcal{X}^d_{n-1}=0\}}]\tilde{G}^{*}_{ik}(s)[\tilde{\bold{G}}(s)^{m-2}]_{kj}\nonumber\\
&+\sum_{j=1}^{N}\sum_{i=1}^{N}\E[e^{-s\mathcal{W}^{sc}_n}1_{\{\mathcal{J}_n=i\}}1_{\{\mathcal{X}^d_{n-1}\geq 1\}}][\tilde{\bold{G}}(s)^{m-1}]_{ij},\quad m \geq 2,\label{Ch5:Waiting_W_m}\\
\E[e^{-sS^{(m)}}]&=\E[e^{-sW^{(m+1)}}],\quad m \geq 1. \label{Ch5:Sojourn_S_m}
\end{align}

Now, we are interested in the probability of being  the $m$-th customer served in a batch. For that, we define the arriving batch-size probabilities as  $b_k=\P(B=k)$ for $k\geq 1$. Therefore, the probability that an arbitrary customer arrives in a batch of size $k$, is equal to  $\frac{kb_k}{\E[B]}$ (see Burke \cite{Burke}). And hence, the probability of being  the $m$th customer served in a batch is given by
\begin{align}
r_m=\sum_{k=m}^{\infty}\frac{kb_k}{\E[B]} \frac{1}{k}=\frac{1}{{\E[B]}}\sum_{k=m}^{\infty}b_k.\label{Prob_mth_batch}
\end{align}

Hence, the steady-state waiting and sojourn time LST of an arbitrary customer are given by
\begin{align}
\E[e^{-sW}]&=\sum_{m=1}^{\infty}r_m\E[e^{-sW^{(m)}}],\label{waitin_arbitrary}\\
\E[e^{-sS}]&=\sum_{m=1}^{\infty}r_m\E[e^{-sS^{(m)}}].\label{sojourn_arbitrary}
\end{align}
\begin{remark}
In case that batches have a maximum size of, say, $M$, we can still use Equations \eqref{Ch5:Waiting_W_m} and \eqref{Ch5:Sojourn_S_m}. However, we note that although we define $\E[e^{-sW^{(m)}}]$ for $m=1,2,...,M+1$, there is in fact no $(M+1)$-th customer in the batch. Still, we need to define $\E[e^{-sW^{(M+1)}}]$ to determine $\E[e^{-sS^{(M)}}]$. Alternatively, one can use $\E[e^{-sS^{(M)}}]=\E[e^{-s\mathcal{S}^{sc}}]$.
\end{remark}

\section{Applications to road traffic}\label{App to road traffic}
The queueing model considered in this paper arises in several applications including logistics, production/inventory systems, computer and telecommunication networks. In this section, we focus on the application to road traffic situations involving multiple conflicting traffic streams. More specifically, we consider an unsignalized priority-controlled intersection used by two traffic streams, both of which wish to cross the intersection (see Fig. \ref{fig:intersection}). There are two priorities: the car drivers on the major road have priority over the car drivers on the minor road (and hence do not experience any impact from the car drivers on the minor road).
The low-priority car drivers, on the minor road, cross the intersection as soon as they come across a gap with duration larger than $T$ between two subsequent high-priority cars, commonly referred to as the \emph{critical headway}. On the major road, we consider Markov platooning (see also \cite{AbhishekMMPPandImpatience}) which can be used to model the fluctuations in the traffic density with a dependency between successive gap sizes.\\

\begin{figure}[th]
\begin{center}
\includegraphics[width=0.7\linewidth]{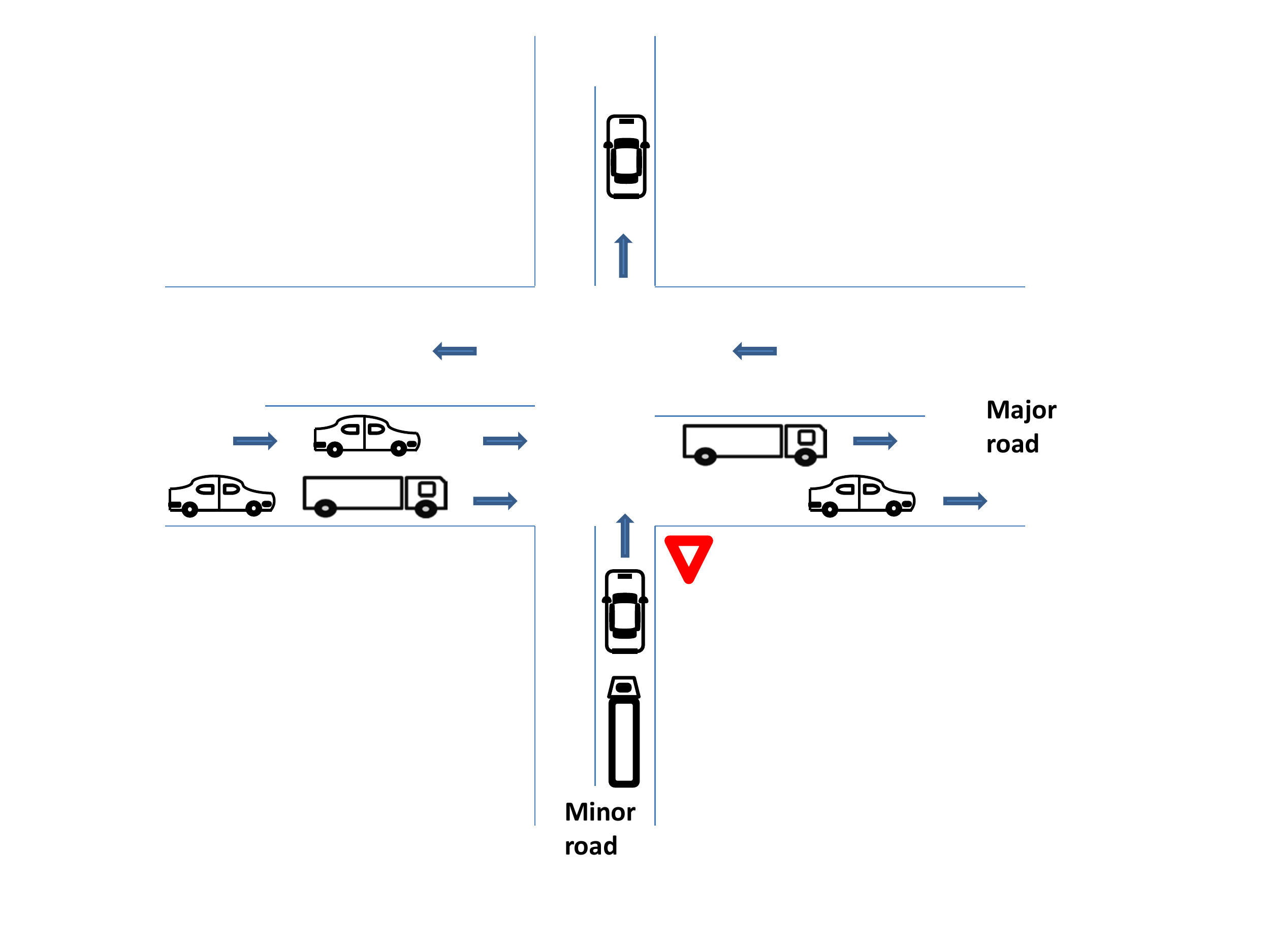}
\end{center}
\caption{An example of an unsignalized intersection considered in this paper.}
\label{fig:intersection}
\end{figure}

On the minor road, cars arrive in batches of size $B$, with PGF $B(z)$, according to a Poisson process with rate $\lambda$. The arrival process on the major road is a Markov modulated Poisson process (MMPP) such that, for $i=1,2,\dots,N$, $q_i$ is the Poisson rate when the continuous time Markov process (so-called background process), $J(t)$, is in phase $i$.
By introducing Markov platooning, an arrival process based on Markov modulation, we create a new, refined way of bunching on the major road. The semi-Markovian service times allow us to capture the required dependence between successive gap sizes.
Platoon forming is a phenomenon that is frequently encountered in practice.
Wu \cite{wu2001} distinguishes between four different traffic flow regimes: free space (no vehicles), free flow (single vehicles), bunched traffic (platoons of vehicles), and queueing. In modern traffic manuals, it is suggested that intersection performance characteristics (such as capacity, which is the reciprocal of the mean service time) can be obtained by analyzing the intersection in one specific regime, and taking weighted averages of the steady-state performance measure under each of the regimes. However, this approach may lead to severe errors and it shown in \cite{AbhishekMMPPandImpatience} that one should build one model that  captures all the variations in traffic flow instead.
For this reason, we will show in this section how to use the single server queue with semi-Markovian service times to develop \emph{one} gap acceptance model, capturing multiple traffic flow regimes on the major road by modeling them with a Markovian arrival process. We show how to obtain the service-time distributions of vehicles on the minor road for each of the three driver behavior types (B$_1$, B$_2$, and B$_3$), which can be plugged into the analysis of Section~\ref{sect:queueingmodel} to obtain the queue length PGF and waiting time LST. We will first define in more detail what we mean by service time in this application.\\

\begin{definition}[Service time]
The service time of a vehicle on the minor road is the time between its arrival at the stop line and the moment when it has crossed the major road. The service time consists of two parts:
\begin{itemize}
\item Scanning for a sufficiently large gap on the major road. This scanning time will be zero in case the remaining time until the next vehicle on the major road arrives is greater than the critical gap $T$;
\item Crossing the road, while freeing up the space for the next car to start scanning. This time is assumed to be equal to $T$, i.e. exactly the size of the critical gap.
\end{itemize}
\end{definition}


The transition probabilities of the background process of the MMPP  are given by
\begin{align*}
\mathbb{P}(J(T)=j|J(0)=i)=[e^{TQ}]_{ij}, \quad  \text{for } i,j=1,2,\dots,N,
\end{align*} with  transition rate matrix  $$Q=\begin{bmatrix}
\mu_{11} & \mu_{12}&\dots &\mu_{1N}\\
\mu_{21}& \mu_{22}& \dots &\mu_{2N}\\
\vdots &\vdots &\dots & \vdots\\
\mu_{N1}& \mu_{N2}& \dots& \mu_{NN}
\end{bmatrix},$$ where $-\mu_{ii}=\mu_{i}=\sum_{j\neq i}\mu_{ij}$.\\

Let $J_n$ and $\bar{J}_n$ be the phase on the major road, seen by the $n$-th low priority car at the \emph{beginning of its service} when the $(n-1)$-th car left the system non empty and empty respectively. In other words, we can say that $J_n$ is the phase on the major road when the $(n-1)$-th car has crossed the major road.
We can write
$$A^{*}_{ij}(z)=\sum_k \bar{P}_{ik}A_{kj}(z),$$ where $\bar{P}_{ik}=\P(\bar{J}_n=k|J_n=i,X_{n-1}=0)$ which is given by\\

$$\bar{P}_{ik}=\frac{\lambda}{\lambda+\mu_i}1_{\{k=i\}}+\frac{\mu_i}{\lambda+\mu_i}\sum_{l\neq i}\frac{\mu_{il}}{\mu_i}\bar{P}_{lk}.$$\\
This implies that \\
$$\lambda \bar{P}_{ik}-\sum_{l}\mu_{il}\bar{P}_{lk}=\lambda1_{\{k=i\}}.$$
We can write  this in matrix form as \\
$$\lambda \bar{P}-Q\bar{P}=\text{diag}(\lambda),$$
and hence we obtain $$\bar{P}=(I-\frac{1}{\lambda}Q)^{-1}.$$
Now we determine the LST of the service time distribution for each of the three types of driver behavior.

\subsection{$\mathbf{B_1}$ (Constant critical gap)} Every driver on the minor road needs the same constant critical headway $T$ to enter the major road. Denote by $G^{(n)}$ the service time of the $n$-th minor road car and $J(t)$ the phase seen by the low priority car driver on the major road at time $t$. We define, for $i,j=1,2,\dots,N$,
 \begin{align}
 G_{ij}(x)&=\mathbb{P}(G^{(n)}\leq x,J_{n+1}=j|J_n=i, X_{n-1}\geq 1)\nonumber\\
 &=\mathbb{P}(G^{(n)}\leq x, J(G^{(n)})=j|J(0)=i),\\
\tilde{G}_{ij}(s)&=\E[e^{-sG^{(n)}}1_{\{J_{n+1}=j\}}|J_n=i, X_{n-1}\geq 1].
\end{align}

Now, firstly, we determine the probability that there is no car on the major road in $[0,T]$ and
$J(T)=j$, given that $J(0)=i$. For that we define
$u_i(t)=\int_{u=0}^{t}1_{\{J(u)=i\}}{\rm d}u$, with $\sum_{i=1}^{N}u_i(t)=t$, and
\begin{align}
\phi_{ij}(t)&=
\mathbb{P}(\text{No car on the major road in $[0,t]$ and $J(t)=j$}|J(0)=i)\label{P_No_car}\\
&=e^{-q_i t}e^{-\mu_i t}1_{\{i=j\}}+\int_{u=0}^{t}\mu_ie^{-\mu_i u}e^{-q_i u}\sum_{k \neq i}\frac{\mu_{ik}}{\mu_i}\phi_{kj}(t-u){\rm d}u\nonumber\\
&=e^{-(q_i+\mu_i) t}1_{\{i=j\}}+\int_{u=0}^{t}e^{-(\mu_i+q_i) u}\sum_{k \neq i}\mu_{ik}\phi_{kj}(t-u){\rm d}u;\nonumber\\
\psi_{ij}(t)&=\begin{cases}
\phi_{ij}(t)q_i & \text{ if } i = j,\\
\phi_{ij}(t)    & \text{ if } i \neq j.
\end{cases}\label{psiij}
\end{align}

We now present in Theorem \ref{thm1} the service-time LST of vehicles on the minor road.

\begin{theorem}\label{thm1}
The LST of the conditional service time $\tilde{G}_{ij}(s)$ for behavior type B$_1$ is the solution to the following system of equations:
\begin{align}
\tilde{G}_{ij}(s)&=e^{-sT}\phi_{ij}(T)+\sum_{k=1}^{N}\tilde{G}_{kj}(s)\int_{t=0}^{T}e^{-st}\psi_{ik}(t){\rm d}t,\label{G_ij_con}
\end{align}\\
for $i,j=1,2,\dots,N$.
\end{theorem}

\begin{proof}
First we solve the system of equations for $\phi_{ij}(t)$ by taking its Laplace-Stieltjes transform:
\begin{align*}
&\tilde{\phi}_{ij}(\omega)=\int_{t=0}^{\infty}e^{-\omega t}\phi_{ij}(t){\rm d}t\\
&=\frac{1}{\omega+\mu_i+q_i}1_{\{i=j\}}+\frac{1}{\omega+\mu_i+q_i}\sum_{k\neq i}\mu_{ik}\tilde{\phi}_{kj}(\omega)\\
&=\frac{1}{\omega+\mu_i+q_i}1_{\{i=j\}}+\frac{1}{\omega+\mu_i+q_i}\sum_{k =1}^{N}\mu_{ik}\tilde{\phi}_{kj}(\omega)+\frac{\mu_i}{\omega+\mu_i+q_i}\tilde{\phi}_{ij(\omega)}.\\
\end{align*}
This implies that, for  $i,j=1,\dots N$,
\begin{align}
\frac{\omega+q_i}{\omega+\mu_i+q_i}\tilde{\phi}_{ij(\omega)}&=\frac{1}{\omega+\mu_i+q_i}\left(1_{\{i=j\}}+\sum_{k =1}^{N}\mu_{ik}\tilde{\phi}_{kj}(\omega)\right).
\end{align}
We can write the above system of equations in matrix form as
\begin{align*}
\text{diag}\left(\frac{\omega+q_i}{\omega+\mu_i+q_i}\right)\tilde{\phi}(\omega)=\text{diag}\left(\frac{1}{\omega+\mu_i+q_i}\right)\Big(I+Q\tilde{\phi}(\omega)\Big),
\end{align*} where $\text{diag}(d_i)=\text{diag}(d_1,d_2,\dots,d_N)$ is a diagonal matrix and $\tilde{\phi}(\omega)=[\tilde{\phi}_{ij}(\omega)]_{N\times N}.$

After simplification, we obtain $\tilde{\phi}(\omega)$ as
\begin{align}
\tilde{\phi}(\omega)=\left(I-\text{diag}\left(\frac{1}{\omega+q_i}\right)Q\right)^{-1}\text{diag}\left(\frac{1}{\omega+q_i}\right). \label{phi_o}
\end{align}
We readily find $\phi_{ij}(t)=\mathcal{L}^{-1}(\tilde{\phi}_{ij}(\o))$ for $i,j=1,2,\dots,N$, where $\mathcal{L}^{-1}$ is the inverse Laplace–Stieltjes transform operator.

Now, we need to determine the probability that at least one car arrives on the major road before time $T$. Let $T_{\text{next car}}$ be the time when the next car passes on the major road and $\psi_{ij}(t)=\mathbb{P}(T_{\text{next car}}\leq t, J(T_\text{next car})=j|J(0)=i)$. We will show that $\psi_{ij}(t)$ satisfies \eqref{psiij} by also taking its transform:
\begin{align*}
\tilde{\psi}_{ij}(\omega)&=\mathbb{E}[e^{-\o T_{\text{next car}}}1_{\{J(T_\text{next car})=j\}}|J(0)=i]\\
&=\frac{\mu_i+q_i}{\o+\mu_i+q_i}\Big(\frac{\mu_i}{\mu_i+q_i} \sum_{k\neq i}\frac{\mu_{ik}}{\mu_i}\tilde{\psi}_{kj}(\o)+\frac{q_i}{\mu_i+q_i}1_{\{i=j\}}\Big).
\end{align*}
After simplification, we can write this  as
\begin{align*}
(\o +q_i)\tilde{\psi}_{ij}(\omega)=\sum_{k=1}^{N}\mu_{ik}\tilde{\psi}_{kj}(\omega)+q_i 1_{\{i=j\}} \text{ for }i,j=1,\dots, N,
\end{align*} and hence, in matrix form as
\begin{align*}
\text{diag}(\o +q_i)\tilde{\psi}(\omega)=Q\tilde{\psi}(\omega)+\text{diag}(q_i),
\end{align*} where $\tilde{\psi}(\omega)=[\tilde{\psi}_{ij}(\omega)]_{N\times N}$.\\

Therefore, we obtain  $\tilde{\psi}(\omega)$ as
\begin{align}
\tilde{\psi}(\omega)=\left(I-\text{diag}\left(\frac{1}{\omega+q_i}\right)Q\right)^{-1}\text{diag}\left(\frac{q_i}{\omega+q_i}\right). \label{psi_o}
\end{align}
 From Equations \eqref{phi_o} and \eqref{psi_o}, we conclude the following relation
 \begin{align}
 \tilde{\psi}(\omega)=\tilde{\phi}(\omega)\text{diag}(q_i). \label{rel : phi_psi}
 \end{align}
As a result, we obtain after taking the inverse Laplace-Stieltjes transform,
\begin{align}
 \psi(t)=\phi(t)\text{diag}(q_i). \label{rel : phi_psi_t}
 \end{align}
This leads to the conditional service-time LST
\begin{align}
\tilde{G}_{ij}(s)&=e^{-sT}\mathbb{P}(\text{No car on the major road in [0,T] and $J(T)=j$}|J(0)=i)\nonumber\\
&\qquad \qquad \qquad \qquad \qquad \qquad +\int_{t=0}^{T}\sum_{k=1}^{N}\psi_{ik}(t)e^{-st}\tilde{G}_{kj}(s){\rm d}t,
\end{align}\\
for $i,j=1,2,\dots,N$, which can be rewritten to \eqref{G_ij_con}, proving the theorem.
\end{proof}

\textbf{Special case:} Let $N=2$, i.e., the MMPP is having two phases on the major road. In this case, we obtain $\tilde{\phi}(\omega)$ from \eqref{phi_o} as
\begin{align}
\tilde{\phi}(\omega)&=\begin{bmatrix}
1+\frac{\mu_1}{\omega+q_1} & -\frac{\mu_1}{\omega+q_1}\\
-\frac{\mu_2}{\omega+q_2} & 1+\frac{\mu_2}{\omega+q_2}
\end{bmatrix}^{-1}\begin{bmatrix}
\frac{1}{\omega+q_1}& 0\\
0& \frac{1}{\omega+q_2}
\end{bmatrix}\nonumber\\
&=\left(\frac{(\omega+q_1)(\omega+q_2)}{\omega^2+(q_1+\mu_1+q_2+\mu_2)\omega+\mu_1q_2+\mu_2q_1+q_1q_2}\right) \begin{bmatrix}
1+\frac{\mu_2}{\omega+q_2} & \frac{\mu_1}{\omega+q_1}\\
\frac{\mu_2}{\omega+q_2} & 1+\frac{\mu_1}{\omega+q_1}
\end{bmatrix}\begin{bmatrix}
\frac{1}{\omega+q_1}& 0\\
0& \frac{1}{\omega+q_2}
\end{bmatrix}\nonumber\\
&=\left(\frac{1}{\omega^2+(q_1+\mu_1+q_2+\mu_2)\omega+\mu_1q_2+\mu_2q_1+q_1q_2}\right) \begin{bmatrix}
\omega+q_2+\mu_2 & \mu_1\\
\mu_2 & \omega+q_1+\mu_1
\end{bmatrix}. \label{phi12}
\end{align}
Now, firstly, we determine the zeros (say $\omega_1, \omega_2$) of the polynomial $\omega^2+(q_1+\mu_1+q_2+\mu_2)\omega+\mu_1q_2+\mu_2q_1+q_1q_2$ which are given by\begin{align}
\omega=\frac{-(q_1+\mu_1+q_2+\mu_2) \pm \sqrt{q_1^2+q_2^2+\mu_1^2+\mu_2^2+2q_1\mu_1+2\mu_1\mu_2+2q_2\mu_2-2\mu_1q_2-2q_1\mu_2-2q_1q_2}}{2}. \label{omega_zeros}
\end{align}
 From Equation \eqref{omega_zeros}, we observe that the zeros $\omega_1$ and $\omega_2$ are real, distinct and non-positive. Moreover, without loss of generality, we assume that $\omega_1>\omega_2$. \\

  Therefore, we can write Equation \eqref{phi12} as
\begin{align*}
\tilde{\phi}(\omega) &=\begin{bmatrix}
\frac{\omega+q_2+\mu_2}{(\omega-\omega_1)(\omega-\omega_2)} & \frac{\mu_1}{(\omega-\omega_1)(\omega-\omega_2)}\\
\frac{\mu_2}{(\omega-\omega_1)(\omega-\omega_2)} & \frac{\omega+q_1+\mu_1}{(\omega-\omega_1)(\omega-\omega_2)}
\end{bmatrix}.
\end{align*}
 After partial fractions, we obtain
 \begin{align*}
\tilde{\phi}(\omega) &=\frac{1}{\omega_1-\omega_2}\begin{bmatrix}
\frac{\omega_1+q_2+\mu_2}{\omega-\omega_1}-\frac{\omega_2+q_2+\mu_2}{\omega-\omega_2} & \frac{\mu_1}{\omega-\omega_1}-\frac{\mu_1}{\omega-\omega_2}\\
\frac{\mu_2}{\omega-\omega_1}-\frac{\mu_2}{\omega-\omega_2} & \frac{\omega_1+q_1+\mu_1}{\omega-\omega_1}-\frac{\omega_2+q_1+\mu_1}{\omega-\omega_2}
\end{bmatrix}.
\end{align*}
 After taking the inverse Laplace transformation,
the elements $\phi_{ij}(t)$ of the matrix $\phi(t)$ are given by
\begin{align}
\phi_{ij}(t)=\begin{cases}
\frac{\mu_i}{\omega_1-\omega_2}(e^{\omega_1t}-e^{\omega_2t}),  i\neq j\\
\frac{1}{\omega_1-\omega_2}\Big((\omega_1+q_{3-i}+\mu_{3-i})e^{\omega_1t}-(\omega_2+q_{3-i}+\mu_{3-i})e^{\omega_2t}\Big), i=j
\end{cases} \label{phi2ij_t}
\end{align}
From Equation \eqref{rel : phi_psi_t}, we obtain the following relations
\begin{align}
\psi_{ij}(t)=q_j\phi_{ij}(t), ~~~~~~\text{for $i,j=1,2.$}
\end{align}  \\

Now, we know  the expressions for $\psi_{ij}(t)$ and $\phi_{ij}(t)$ which we need to determine the LST of the conditional service time. For $N=2$, Equation \eqref{G_ij_con} becomes
\begin{align*}
\tilde{G}_{ij}(s)
&=e^{-sT}\phi_{ij}(T)+\tilde{G}_{1j}(s)\int_{t=0}^{T}e^{-st}\psi_{i1}(t){\rm d}t+\tilde{G}_{2j}(s)\int_{t=0}^{T}e^{-st}\psi_{i2}(t){\rm d}t.
\end{align*}\\
For $i=1$, after substituting the values of $\psi_{ij}$, we obtain the above expression as
 \begin{align*}
\tilde{G}_{1j}(s)&=e^{-sT}\phi_{1j}(T)+\tilde{G}_{1j}(s)\int_{t=0}^{T}e^{-st}\frac{q_1}{\omega_1-\omega_2}\Big((\omega_1+q_{2}+\mu_{2})e^{\omega_1t}\nonumber\\
&-(\omega_2+q_{2}+\mu_{2})e^{\omega_2t}\Big){\rm d}t
+\tilde{G}_{2j}(s)\int_{t=0}^{T}e^{-st}\frac{\mu_1q_2}{\omega_1-\omega_2}(e^{\omega_1t}-e^{\omega_2t}){\rm d}t.
\end{align*}
After simplification, we can write this as
\begin{align}
&\Bigg( 1-\frac{q_1}{\omega_1-\omega_2}\Big( \frac{(\omega_1-\omega_2)(s+q_2+\mu_2)}{(s-\omega_1)(s-\omega_2)}-\frac{\omega_1+q_2+\mu_2}{s-\omega_1}e^{-(s-\omega_1)T}+\frac{\omega_2+q_2+\mu_2}{s-\omega_2}e^{-(s-\omega_2)T}\Big)\Bigg)\tilde{G}_{1j}(s)\nonumber\\
&-\frac{\mu_1q_2}{\omega_1-\omega_2}\Bigg( \frac{\omega_1-\omega_2}{(s-\omega_1)(s-\omega_2)}-\frac{1}{s-\omega_1}e^{-(s-\omega_1)T}+\frac{1}{s-\omega_2}e^{-(s-\omega_2)T}\Bigg)\tilde{G}_{2j}(s)=e^{-sT}\phi_{1j}(T),\label{G2_ij_first}
\end{align} where $\phi_{1j}(T)$ is given by Equation \eqref{phi2ij_t} with $i=1, t=T$.\\

Similarly, for $i=2$, we obtain the following equation in $\tilde{G}_{1j}(s)$ and $\tilde{G}_{2j}(s)$
 \begin{align}
&-\frac{\mu_2q_1}{\omega_1-\omega_2}\Bigg( \frac{\omega_1-\omega_2}{(s-\omega_1)(s-\omega_2)}-\frac{1}{s-\omega_1}e^{-(s-\omega_1)T}+\frac{1}{s-\omega_2}e^{-(s-\omega_2)T}\Bigg) \tilde{G}_{1j}(s) \nonumber\\
&+\Bigg( 1-\frac{q_2}{\omega_1-\omega_2}\Big( \frac{(\omega_1-\omega_2)(s+q_1+\mu_1)}{(s-\omega_1)(s-\omega_2)}-\frac{\omega_1+q_1+\mu_1}{s-\omega_1} e^{-(s-\omega_1)T} \nonumber\\
&+\frac{\omega_2+q_1+\mu_1}{s-\omega_2}e^{-(s-\omega_2)T}\Big)\Bigg)\tilde{G}_{2j}(s)=e^{-sT}\phi_{2j}(T),\label{G2_ij_second}
\end{align} where $\phi_{2j}(T)$ is given by Equation \eqref{phi2ij_t} with $i=2, t=T$.\\
Hence, we have two linear equations  with two unknowns $\tilde{G}_{1j}(s)$ and $\tilde{G}_{2j}(s)$ for fixed $j$, which we can solve to obtain $\tilde{G}_{1j}(s)$ and $\tilde{G}_{2j}(s)$ for fixed $j$.

\subsection{$\mathbf{B_2}$ (Inconsistent behavior)} Now we assume that every car driver samples a random $T$ for each new `attempt'. An new attempt starts whenever a car on the major road passes that was too close to its predecessor, not leaving a gap that was large enough for the car on the minor road to cross.
Using \eqref{P_No_car}, for $i,j=1,2,\dots,N$, the LST of the conditional service time is given by
\begin{align}
\tilde{G}_{ij}(s)&=\E[e^{-sG^{(n)}}1_{\{J_{n+1}=j\}}|J_n=i, X_{n-1}\geq 1]\nonumber\\
=&\E\Big[e^{-sT}\mathbb{P}(\text{No car on the major in [0,T] and $J(T)=j$}|J(0)=i)\nonumber\\
&\qquad \qquad \qquad \qquad \qquad \qquad+\int_{t=0}^{T}\sum_{k=1}^{N}\psi_{ik}(t)e^{-st}\tilde{G}_{kj}(s){\rm d}t\Big]\nonumber\\
=&\E[e^{-sT}\phi_{ij}(T)]+\sum_{k=1}^{N}\tilde{G}_{kj}(s)\E\Big[\int_{t=0}^{T}\psi_{ik}(t)e^{-st}{\rm d}t\Big].\label{G_ij_rand}
\end{align}
Now we have $N^2$ linear equations for $\tilde{G}_{ij}(s)$. The solution of this system of equations provides $\tilde{G}_{ij}(s)$.\\
\newline
\subsection{$\mathbf{B_3}$ (Consistent behavior)} Every car driver samples a random $T$ at his first attempt and this (random) value will be used consistently for each new attempt by this driver. Let $G^{(T,n)}$ be the service time  in model $B_1$ of the $n$-th low-priority vehicle, conditional on having a deterministic critical gap $T$.  Then for $i,j=1,2,\dots,N$, the LST of the conditional service time for a vehicle in model B$_3$ is given by \\
\begin{align}
\tilde{G}_{ij}(s)=\E[e^{-sG^{(n)}}1_{\{J_{n+1}=j\}}|J_n=i, X_{n-1}\geq 1]=\E\Big[\tilde{G}_{ij}(T,s)\Big], \label{G_ij_same}
\end{align}
where $\tilde{G}_{ij}(T,s)=\E[e^{-sG^{(T,n)}}1_{\{J_{n+1}=j\}}|J_n=i, X_{n-1}\geq 1]$ which we obtain from \eqref{G_ij_con} as \begin{align}
\tilde{G}_{ij}(T,s)&=e^{-sT}\phi_{ij}(T)+\sum_{k=1}^{N}\tilde{G}_{kj}(T,s)\int_{t=0}^{T}e^{-st}\psi_{ik}(t){\rm d}t.\label{G_ij_con_same}
\end{align}
In this model, firstly, $\tilde{G}_{ij}(T,s)$ is determined from the system of equations \eqref{G_ij_con_same}. We then obtain the LST of the conditional service time, $\tilde{G}_{ij}(s)$, from Equation \eqref{G_ij_same}.

\section{Numerical results}\label{numericalresults:ch5}
In this section, we present some numerical examples to demonstrate the impact of batch arrivals and Markov platooning on the delay on the minor road. For simplicity, we restrict ourselves to two phases of the background process of the MMPP on the major road ($N = 2$), corresponding to high and low traffic intensities.

\subsection{Example 1: the impact of batch arrivals} In this example, we compare the expected waiting times on the minor road, for the three behavior types B$_1$, B$_2$, and B$_3$. With behavior type B$_2$, we assume that a minor road driver samples the critical gap (headway) of $6.22$ seconds with probability $0.9$ and $14$ seconds with probability $0.1$, at each new attempt. In model B$_3$, drivers are consistent and keep the same (random) critical gap. In this situation, $90$\% of the minor road drivers need a gap of at least $6.22$ seconds;  the other $10$\% need at least $14$ seconds. For behavior type B$_1$, we take the critical gap $T=6.22\times 0.9+14\times 0.1=7$ seconds. On the major road, vehicles arrive according to rate $q_i$ (veh/h) in phase $i$, for $i=1,2$, with the fixed ratio $q_1=3q_2$, where the background process of the MMPP  stays exponentially distributed times of, on average, $60$ seconds in phase $1$, and $240$ seconds in phase $2$, i.e., $\mu_1=1/60$ and $\mu_2=1/240$. Therefore, the long-term average arrival rate on the major road is given by
\begin{align}\label{qbar}
\bar{q}:=\frac{q_1/\mu_1+q_2/\mu_2}{1/\mu_1+1/\mu_2}=
\frac{q_1\mu_2+q_2\mu_1}{\mu_1+\mu_2}.
\end{align}
We assume that the batch (platoon) arrival rate on the minor road is $\lambda=50$ (batches per hour). We consider the following two  batch size distributions with the same mean $\E[B]=4$:
 \begin{itemize}
   \item Uniform distribution: $\P(B=k)=\begin{cases}
                                              1/7, & \mbox{if } k=1,2,\dots, 7 \\
                                              0, & \mbox{otherwise}.
                                            \end{cases}$
   \item Low/high distribution: $\P(B=k)=\begin{cases}
                                              1/2, & \mbox{if } k=1 \mbox{ or } k=7 \\
                                              0, & \mbox{otherwise}.
                                            \end{cases}$
\end{itemize}
For the model without batch arrivals on the minor road (i.e., mean batch size $1$), we take the arrival rate as $50\times 4=200$ (veh/hour) to make a fair comparison with the model with batch arrivals. From Figure \ref{paper6:example1}, it is first noticed that the expected waiting times for all three behavior types (denoted by $\E[W_1],\E[W_2]$ and $\E[W_3]$) depend not only on the mean batch size, but also on the full distribution of the batch sizes. Second, batch arrivals on the minor road have a negative effect (compared to the individual arrivals on that road), as a function of the average flow rate on the major road, on the expected waiting times. This is to be expected, because even in the case of no traffic on the major road, vehicles arriving in batches still have to wait for all the vehicles in front of them in the same batch, while individual vehicles will hardly have to wait in this situation. Furthermore, we observe that consistent driver behavior results in the longest waiting times. This is due to the fact that one vehicle requiring a large critical gap will need a very long time before crossing the road, while the resampling in model B$_2$ increases the chances of needing a smaller critical headway after a failed attempt.
\begin{figure}[!ht]
\begin{center}
\includegraphics[width=0.8\linewidth]{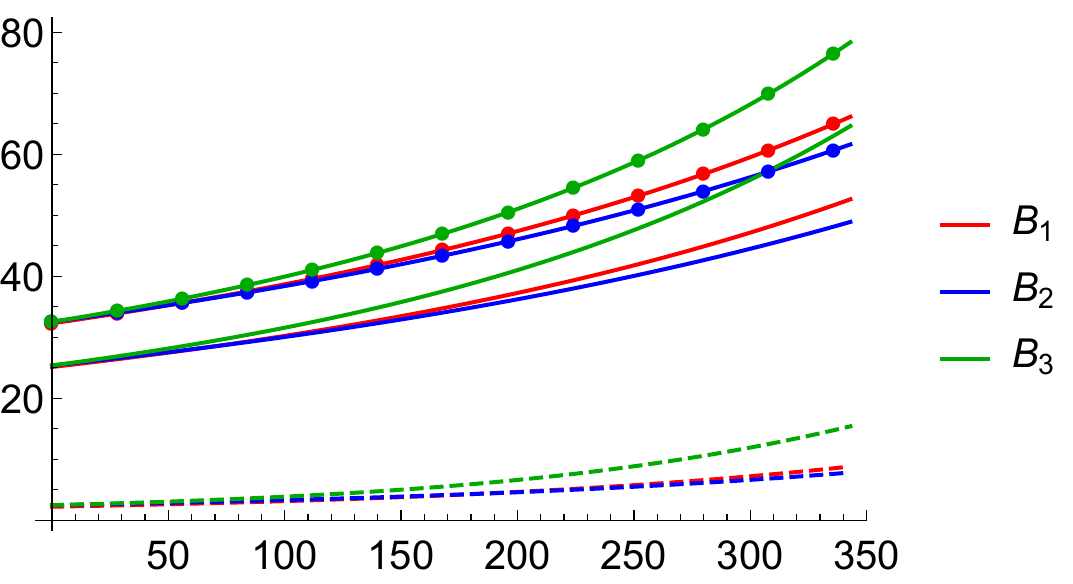}
\end{center}
\caption{Expected waiting times (seconds) of an arbitrary vehicle on the minor road,  as a function of the average flow rate on the major road (veh/h) in Example 1. The solid lines correspond to the model with batch arrivals where the solid lines \emph{with} dots are for the low and high distribution of the batch sizes, and the solid lines \emph{without} dots are for the uniformly distributed batch sizes; the dashed lines correspond to the model without batches.}
\label{paper6:example1}
\end{figure}

For completeness, we have also studied the impact of the batch size distribution, under different traffic intensities on the major road, on the \emph{variance} of the waiting times. For each combination of the three behavior types and batch-size distributions (including no batches) we have computed the mean and variance of the waiting times for $\bar{q}=70$ and $\bar{q}=420$.
The numerical results are shown in Table~\ref{tbl:numericalresults}. The most interesting result here is that the strict ordering for the \emph{mean} waiting times, $\E[W_2]<\E[W_1]<\E[W_3]$, is no longer true for the variance of the waiting times. In the case where batch sizes are uniformly distributed, the variance for model B$_1$ is smaller than the variance for model B$_2$. It might be interesting to try and obtain more insight into this phenomenon, but this is beyond the scope of this paper.

\begin{table}[h!]
\begin{center}
\begin{tabular}{|c|c|c|c|c|}
  \hline
   \multicolumn{5}{|c|}{Low/high distribution} \\
  \hline
   & \multicolumn{2}{|c|}{$\bar{q}=70$}  & \multicolumn{2}{|c|}{$\bar{q}=420$} \\
   \cline{2-5}
   & \multicolumn{1}{|r}{$\E[W]$} & \multicolumn{1}{r|}{Var$(W)$}  & \multicolumn{1}{|r}{$\E[W]$} & \multicolumn{1}{r|}{Var$(W)$} \\
   \hline
  B$_1$ & \multicolumn{1}{|r}{36.55}  & \multicolumn{1}{r|}{1134.62} & \multicolumn{1}{|r}{80.95} & \multicolumn{1}{r|}{6537.15}  \\
  \hline
   B$_2$ & \multicolumn{1}{|r}{28.52}  & \multicolumn{1}{r|}{744.26} & \multicolumn{1}{|r}{64.85} & \multicolumn{1}{r|}{4502.77}  \\
  \hline
 B$_3$ & \multicolumn{1}{|r}{37.58}  & \multicolumn{1}{r|}{1262.58} & \multicolumn{1}{|r}{105.09} & \multicolumn{1}{r|}{12595.68}  \\
  \hline
\multicolumn{5}{c}{}\\[1ex]
%
  \hline
   \multicolumn{5}{|c|}{Uniform distribution} \\
  \hline
   & \multicolumn{2}{|c|}{$\bar{q}=70$}  & \multicolumn{2}{|c|}{$\bar{q}=420$} \\
   \cline{2-5}
   & \multicolumn{1}{|r}{$\E[W]$} & \multicolumn{1}{r|}{Var$(W)$}  & \multicolumn{1}{|r}{$\E[W]$} & \multicolumn{1}{r|}{Var$(W)$} \\
   \hline
  B$_1$ & \multicolumn{1}{|r}{28.52}  & \multicolumn{1}{r|}{744.26} & \multicolumn{1}{|r}{64.85} & \multicolumn{1}{r|}{4502.77}  \\
  \hline
   B$_2$ & \multicolumn{1}{|r}{28.39}  & \multicolumn{1}{r|}{762.28} & \multicolumn{1}{|r}{57.33} & \multicolumn{1}{r|}{3520.27}  \\
  \hline
 B$_3$ & \multicolumn{1}{|r}{29.43}  & \multicolumn{1}{r|}{841.50} & \multicolumn{1}{|r}{86.01} & \multicolumn{1}{r|}{9192.12}  \\
  \hline
\multicolumn{5}{c}{}    \\[1ex]
%
  \hline
   \multicolumn{5}{|c|}{No batches} \\
  \hline
   & \multicolumn{2}{|c|}{$\bar{q}=70$}  & \multicolumn{2}{|c|}{$\bar{q}=420$} \\
   \cline{2-5}
   & \multicolumn{1}{|r}{$\E[W]$} & \multicolumn{1}{r|}{Var$(W)$}  & \multicolumn{1}{|r}{$\E[W]$} & \multicolumn{1}{r|}{Var$(W)$} \\
   \hline
  B$_1$ & \multicolumn{1}{|r}{2.82}  & \multicolumn{1}{r|}{24.12} & \multicolumn{1}{|r}{12.73} & \multicolumn{1}{r|}{453.44}  \\
  \hline
   B$_2$ & \multicolumn{1}{|r}{3.02}  & \multicolumn{1}{r|}{30.24} & \multicolumn{1}{|r}{10.68} & \multicolumn{1}{r|}{327.03}  \\
  \hline
 B$_3$ & \multicolumn{1}{|r}{3.35}  & \multicolumn{1}{r|}{40.38} & \multicolumn{1}{|r}{24.48} & \multicolumn{1}{r|}{1920.96}  \\
  \hline
\end{tabular}
\end{center}
\caption{Means and variances of the waiting times for different batch-size distributions, for each of the three driver behaviors, under low and high traffic volumes on the major road.}
\label{tbl:numericalresults}
\end{table}

\subsection{Example 2: the impact of Markov platooning}
In this example, we take the same settings as in Example 1, but we make two adjustments. First, we
change the critical gaps from $6.22$ to $5$ seconds and from $14$ to $25$ seconds, for the behavior types B$_2$ and B$_3$. The expected critical gap remains $5\times 0.9+25\times 0.1=7$ seconds (which is also the value we take for $T$ in model B$_1$), but the variation is much higher for reasons that will become apparent later. Note that for behavior B$_1$ the distribution is irrelevant and, as such, the results will be the same as in the previous example.
Second, we fix the uniform distribution for the batch sizes on the minor road, as considered in Example 1, where the batch arrival rate is taken as  $\lambda=50$ (batches per hour). For these settings,  we compare the expected waiting times of the model with and without Markov platooning on the major road,  where in the case without platooning, we assume Poisson arrivals on this road,  with rate $\bar{q}$, which can be obtained from Equation \eqref{qbar}. From Figure \ref{paper6:example2}, one can clearly observe that platoon forming has a significant impact on the mean waiting times. However, it is interesting to observe that for small values of $\bar{q}$, platoon forming results in slightly higher waiting times, whereas for large values of $\bar{q}$ platoon forming results in smaller delays on the minor road. Now it also becomes clear why we have taken such extreme (arguably unrealistic) values for the critical gaps ($5$ and $25$ seconds, respectively). Our goal was to magnify the effect of resampling: even more than in Example 1, we can observe that consistent driver behavior (model B$_3$) results in \emph{much} longer waiting times than inconsistent driver behavior (B$_2$) or constant gaps (B$_1$). The reason is that once a vehicle samples a large critical gap of $25$ seconds, it is stuck for a very long time before it can cross the intersection, causing a long queue to build up. Resampling resolves this issue because, in particular when there is much traffic on the high priority road, there will be many attempts in a short period of time and the driver is much more likely to sample a new, sufficiently small critical gap to cross the major road. For this reason, we have decided to plot the mean waiting time in Figure~\ref{paper6:example2} for B$_3$ in a separate figure, because at $q=376.2$ the queue already becomes unstable for the model with Poisson arrivals.

\begin{figure}[!ht]
\parbox{0.49\linewidth}{\centering
\includegraphics[width=\linewidth]{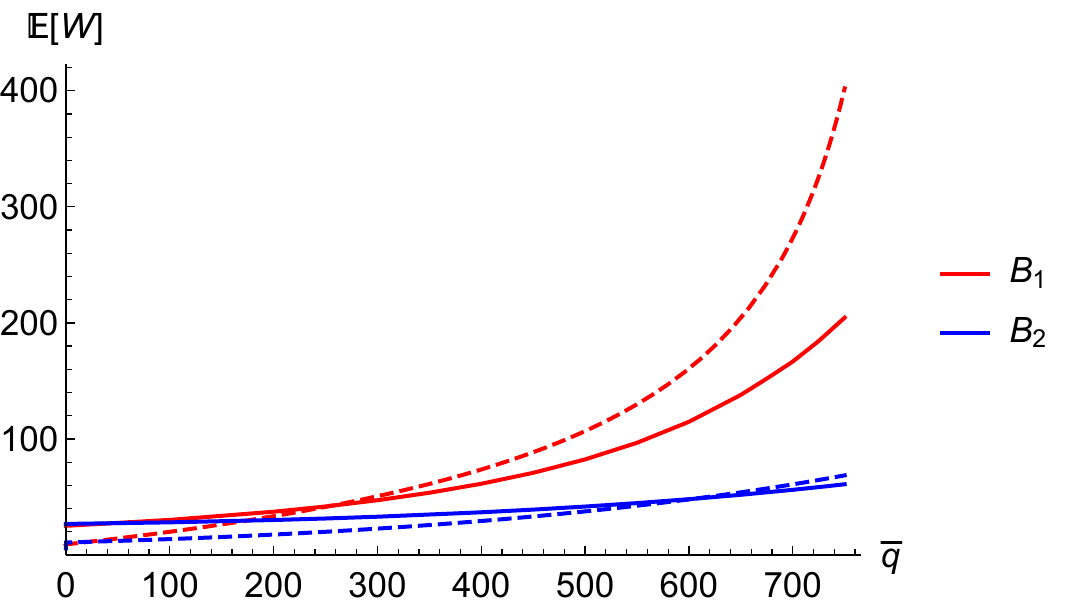}\\
\scriptsize (a)  Expected waiting times for models B$_1$ and B$_2$.}
\parbox{0.49\linewidth}{\centering
\includegraphics[width=\linewidth]{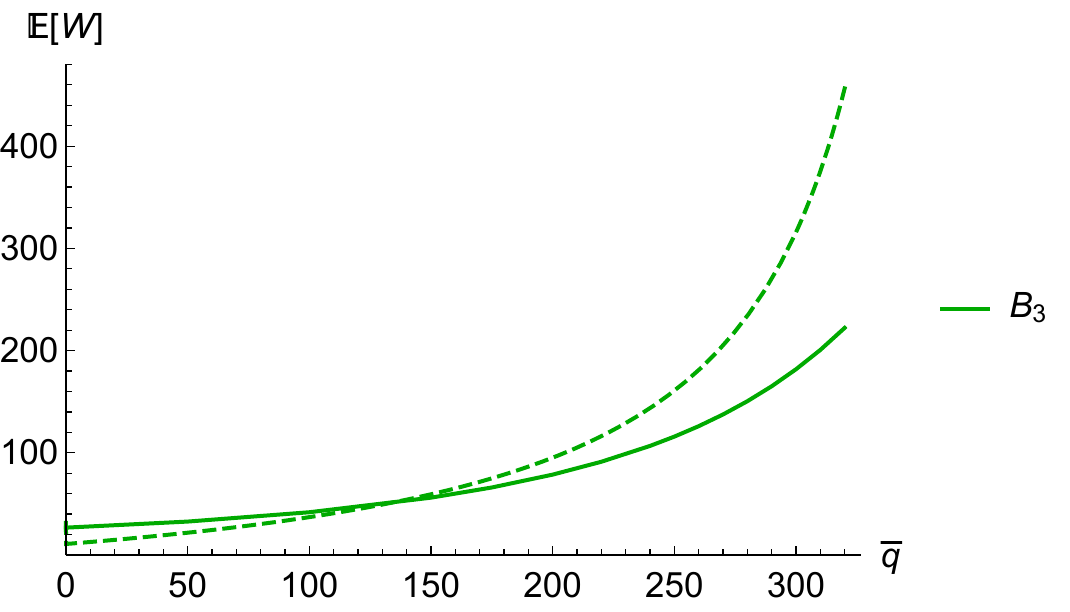}\\
\scriptsize (b) Expected waiting times for model B$_3$.}
\caption{Expected waiting times (seconds) of an arbitrary vehicle on the minor road, as a function of the average flow rate on the major road (veh/h) in Example 2. The solid lines correspond to the model with Markov platooning; the dashed lines correspond to the model without platooning.}
\label{paper6:example2}
\end{figure}

\subsection{Example 3: Approximations}

Computing the sojourn time and waiting time distributions can be quite computationally challenging, in particular when $N$ grows large. It might be useful, for practical purposes, to have a good approximation for the mean queue lengths or the mean waiting/sojourn times. Although it is not the main scope of the present paper, we briefly illustrate how to create an excellent approximation using a well-established technique, interpolating between the light-traffic (LT) and heavy-traffic (HT) limits. For more details about this technique, we refer the reader to \cite{boonapprox2009,reimansimon88,simon92}. The idea is straightforward: from \cite{AbhishekHeavyTraffic} we know the HT limit of the scaled queue length,
\[\lim_{\rho\uparrow1} (1-\rho)\E[X]=1/\eta,\]
with $\eta$ as defined in \cite[Theorem 1]{AbhishekHeavyTraffic}. Moreover, the LT limit of $X$ is easy to find in an intuitive manner. When $\rho$ tends to zero, the system is always empty upon the arrival of a batch, and $\E[X]$ is equal to the mean number of customers (in the same batch) \emph{behind} an arbitrary customer. The argument is similar to \eqref{Prob_mth_batch}; if a customer is the $m$-th in a batch of size $k$, then he leaves behind $k-m$ customers:
\[\delta:=\lim_{\rho\downarrow0} \E[X]=
\sum_{m=1}^\infty\sum_{k=m}^{\infty}\frac{kb_k}{\E[B]} \frac{1}{k}(k-m)
=\sum_{k=1}^\infty k r_{k+1}.
\]
Now we can develop the following approximation for $\E[X]$:
\[
\E[X^\text{approx}]=\frac{\delta+\rho(1/\eta-\delta)}{1-\rho},
\]
which is exact in the limiting cases $\rho\downarrow0$ and $\rho\uparrow1$. An approximation for the mean sojourn time can be found by using the following relation between $\E[X]$ and $\E[S]$ (cf. \cite{AbhishekMergingModel}):
\begin{equation}
\E[S]=\frac{1}{\lambda \E[B]}\left(\E[X]-\frac{B''(1)}{2\E[B]}\right).\label{ewapprox}
\end{equation}
Substituting $\E[X^\text{approx}]$ for $\E[X]$ in \eqref{ewapprox} yields a closed-form approximation for the mean sojourn time, $\E[S^\text{approx}]$, which is remarkably accurate due to its construction.
As an illustration, we select the MMPP model of Example 2 with driver behavior B$_1$, but fixing $\bar{q}=500$ vehicles per hour and varying $\lambda$ instead. The parameter values are
\[
\delta=2, \eta=0.343, \E[B]=4, B''(1)=16, \rho=45.67\lambda.
\]
The resulting approximation \eqref{ewapprox} is plotted in Figure~\ref{fig:approx}, together with the exact results, confirming that this is an excellent approximation. Note that the mean \emph{waiting} time is slightly more difficult to approximate, if one wants it to be exact in light traffic again, due to the fact that the service times are different for customers arriving in an empty system than for customers arriving in a non-empty system. Still, simply taking $\E[W^\text{approx}] = \E[S^\text{approx}] - \E[G^{(n)}|X_{n-1}\geq 1]$ will give very accurate results for moderate to high values of $\rho$.

\begin{figure}[!ht]
\begin{center}
\includegraphics[width=0.7\linewidth]{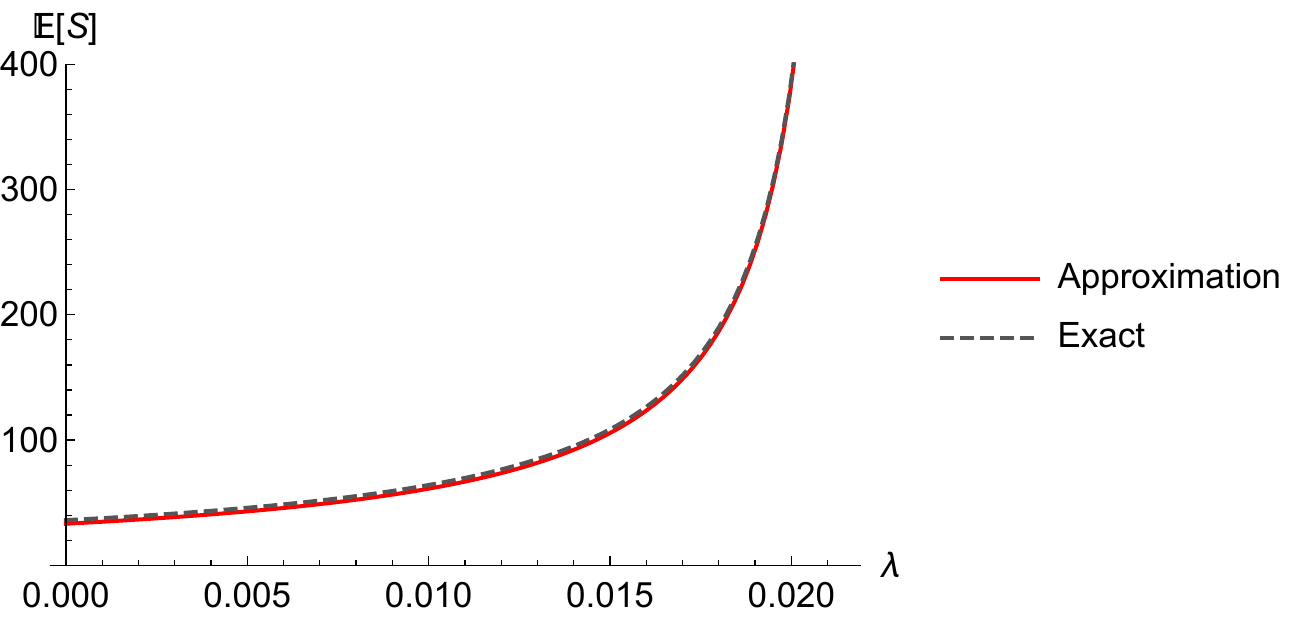}
\caption{Approximated and exact mean sojourn times for Numerical Example 3.}
\label{fig:approx}
\end{center}
\end{figure}

\begin{acknowledgements}
 The research of Abhishek and Rudesindo~N\'u\~nez~Queija is partly funded by NWO Gravitation project {\sc Networks}, grant number 024.002.003. The authors thank Michel Mandjes ( University of Amsterdam) and Onno Boxma ( Eindhoven University of Technology) for helpful discussions.
\end{acknowledgements}

\bibliographystyle{abbrv}

\end{document}